\documentclass[11pt]{amsart}
\usepackage{amssymb}
\usepackage{amsmath}
\usepackage{mathrsfs}
\usepackage{amsfonts}
\usepackage{color}
\usepackage{amsthm}
\usepackage{ esint }
\usepackage{a4}
\usepackage{graphicx}
\usepackage{enumerate}
\usepackage[noadjust]{cite}

\usepackage{setspace}
\usepackage[foot]{amsaddr}

\usepackage{geometry}
\usepackage[T1]{fontenc}
\usepackage[polish,english]{babel}
\usepackage[utf8]{inputenc}
\newtheorem{theorem}{Theorem}

\newtheorem{proposition}{Proposition}

\newtheorem{lemma}{Lemma}

\theoremstyle{definition}
\newtheorem{definition}{Definition}%

\newtheorem{remark}{Remark}
\newtheorem{example}{Example}

\newcommand*{\norm}[1]{\left\Vert{#1}\right\Vert}
\newcommand*{\abs}[1]{\left\vert{#1}\right\vert}

\newcommand*{\Real}{\mathbb{R}}
\newcommand*{\Natural}{\mathbb{N}}

\newcommand*{\Fs}{\mathcal{F}}

\newcommand*{\Hs}{\mathcal{H}}

\newcommand*{\dist}{\mathrm{dist}}

\makeatletter
\let\save@mathaccent\mathaccent
\newcommand*\if@single[3]{%
  \setbox0\hbox{${\mathaccent"0362{#1}}^H$}%
  \setbox2\hbox{${\mathaccent"0362{\kern0pt#1}}^H$}%
  \ifdim\ht0=\ht2 #3\else #2\fi
  }
\newcommand*\rel@kern[1]{\kern#1\dimexpr\macc@kerna}
\newcommand*\widebar[1]{\@ifnextchar^{{\wide@bar{#1}{0}}}{\wide@bar{#1}{1}}}
\newcommand*\wide@bar[2]{\if@single{#1}{\wide@bar@{#1}{#2}{1}}{\wide@bar@{#1}{#2}{2}}}
\newcommand*\wide@bar@[3]{%
  \begingroup
  \def\mathaccent##1##2{%
    \let\mathaccent\save@mathaccent
    \if#32 \let\macc@nucleus\first@char \fi
    \setbox\z@\hbox{$\macc@style{\macc@nucleus}_{}$}%
    \setbox\tw@\hbox{$\macc@style{\macc@nucleus}{}_{}$}%
    \dimen@\wd\tw@
    \advance\dimen@-\wd\z@
    \divide\dimen@ 3
    \@tempdima\wd\tw@
    \advance\@tempdima-\scriptspace
    \divide\@tempdima 10
    \advance\dimen@-\@tempdima
    \ifdim\dimen@>\z@ \dimen@0pt\fi
    \rel@kern{0.6}\kern-\dimen@
    \if#31
      \overline{\rel@kern{-0.6}\kern\dimen@\macc@nucleus\rel@kern{0.4}\kern\dimen@}%
      \advance\dimen@0.4\dimexpr\macc@kerna
      \let\final@kern#2%
      \ifdim\dimen@<\z@ \let\final@kern1\fi
      \if\final@kern1 \kern-\dimen@\fi
    \else
      \overline{\rel@kern{-0.6}\kern\dimen@#1}%
    \fi
  }%
  \macc@depth\@ne
  \let\math@bgroup\@empty \let\math@egroup\macc@set@skewchar
  \mathsurround\z@ \frozen@everymath{\mathgroup\macc@group\relax}%
  \macc@set@skewchar\relax
  \let\mathaccentV\macc@nested@a
  \if#31
    \macc@nested@a\relax111{#1}%
  \else
    \def\gobble@till@marker##1\endmarker{}%
    \futurelet\first@char\gobble@till@marker#1\endmarker
    \ifcat\noexpand\first@char A\else
      \def\first@char{}%
    \fi
    \macc@nested@a\relax111{\first@char}%
  \fi
  \endgroup
}

\newcommand*{\mint}[1]{%
  \mint@l{#1}{}%
}
\newcommand*{\mint@l}[2]{%
  \@ifnextchar\limits{%
    \mint@l{#1}%
  }{%
    \@ifnextchar\nolimits{%
      \mint@l{#1}%
    }{%
      \@ifnextchar\displaylimits{%
        \mint@l{#1}%
      }{%
        \mint@s{#2}{#1}%
      }%
    }%
  }%
}
\newcommand*{\mint@s}[2]{%
  \@ifnextchar_{%
    \mint@sub{#1}{#2}%
  }{%
    \@ifnextchar^{%
      \mint@sup{#1}{#2}%
    }{%
      \mint@{#1}{#2}{}{}%
    }%
  }%
}
\def\mint@sub#1#2_#3{%
  \@ifnextchar^{%
    \mint@sub@sup{#1}{#2}{#3}%
  }{%
    \mint@{#1}{#2}{#3}{}%
  }%
}
\def\mint@sup#1#2^#3{%
  \@ifnextchar_{%
    \mint@sub@sup{#1}{#2}{#3}%
  }{%
    \mint@{#1}{#2}{}{#3}%
  }%
}
\def\mint@sub@sup#1#2#3^#4{%
  \mint@{#1}{#2}{#3}{#4}%
}
\def\mint@sup@sub#1#2#3_#4{%
  \mint@{#1}{#2}{#4}{#3}%
}
\newcommand*{\mint@}[4]{%
  \mathop{}%
  \mkern-\thinmuskip
  \mathchoice{%
    \mint@@{#1}{#2}{#3}{#4}%
        \displaystyle\textstyle\scriptstyle
  }{%
    \mint@@{#1}{#2}{#3}{#4}%
        \textstyle\scriptstyle\scriptstyle
  }{%
    \mint@@{#1}{#2}{#3}{#4}%
        \scriptstyle\scriptscriptstyle\scriptscriptstyle
  }{%
    \mint@@{#1}{#2}{#3}{#4}%
        \scriptscriptstyle\scriptscriptstyle\scriptscriptstyle
  }%
  \mkern-\thinmuskip
  \int#1%
  \ifx\\#3\\\else_{#3}\fi
  \ifx\\#4\\\else^{#4}\fi  
}
\newcommand*{\mint@@}[7]{%
  \begingroup
    \sbox0{$#5\int\m@th$}%
    \sbox2{$#5\int_{}\m@th$}%
    \dimen2=\wd0 %
    \let\mint@limits=#1\relax
    \ifx\mint@limits\relax
      \sbox4{$#5\int_{\kern1sp}^{\kern1sp}\m@th$}%
      \ifdim\wd4>\wd2 %
        \let\mint@limits=\nolimits
      \else
        \let\mint@limits=\limits
      \fi
    \fi
    \ifx\mint@limits\displaylimits
      \ifx#5\displaystyle
        \let\mint@limits=\limits
      \fi
    \fi
    \ifx\mint@limits\limits
      \sbox0{$#7#3\m@th$}%
      \sbox2{$#7#4\m@th$}%
      \ifdim\wd0>\dimen2 %
        \dimen2=\wd0 %
      \fi
      \ifdim\wd2>\dimen2 %
        \dimen2=\wd2 %
      \fi
    \fi
    \rlap{%
      $#5
        \vcenter{%
          \hbox to\dimen2{%
            \hss
            $#6{#2}\m@th$%
            \hss
          }%
        }%
      $
    }%
  \endgroup
}
\newcommand*{\vint}[1]{\mint{-}\limits_{#1}}

\title{Characterization of mean value harmonic functions on~norm~induced~metric~measure spaces with~weighted~Lebesgue~measure}
\author{Antoni Kijowski\textsuperscript{\dag}}
\address{\textsuperscript{\dag}Institute of Mathematics, Polish Academy of Sciences, ul. Śniadeckich 8, 00-656 Warsaw, Poland}
\email{akijowski@impan.pl}
\date{}

\begin{document}

\begin{abstract}
We study the mean-value harmonic functions on open subsets of $\mathbb{R}^n$ equipped with weighted Lebesgue measures and norm induced metrics. Our main result is a necessary condition saying that all such functions solve a certain homogeneous system of elliptic PDEs. Moreover, a converse result is established in case of analytic weights. Assuming Sobolev regularity of weight $w \in W^{l,\infty}$ we show that strongly harmonic functions are as well in $W^{l,\infty}$ and that they are analytic, whenever the weight is analytic.

The analysis is illustrated by finding all mean-value harmonic functions in $\Real^2$ for the $l^p$-distance \mbox{$1 \leq p \leq \infty$}. The essential outcome is a certain discontinuity with respect to $p$, i.e. that for all $p \ne 2$ there are only finitely many linearly independent mean-value harmonic functions, while for $p=2$ there are infinitely many of them. We conclude with a remarkable observation that strongly harmonic functions in~$\mathbb{R}^n$ possess the mean value property with respect to infinitely many weight functions obtained from a given weight.

\mbox{}

\mbox{}
\end{abstract}

\maketitle
\section{Introduction}
\pagestyle{plain}
Analysis on metric spaces has been intensively developed through the last two decades. Studies of such researchers as Cheeger, Hajłasz, Heinonen, Koskela and Shanmugalingam brought new light to a notion of the gradient in metric measure spaces. One of many important notions of this area is a counterpart of a harmonic function on metric measure spaces being a minimizer of the Dirichlet energy. Recently, there has been a new approach to this topic by using the mean value property. Such an approach is much easier to formulate, than the variational one, because it does not require the notion of the Sobolev spaces on metric measure spaces. Strongly and weakly harmonic were introduced in \cite{J} and~\cite{H} by Adamowicz, Gaczkowski and Górka. Authors developed the theory of such functions providing e.g. the Harnack inequality, the H\"older and Lipschitz regularity results and studying the Perron method. Nevertheless, many questions remain unanswered, including the one on the relation between minimizers of the Dirichlet energy and mean value harmonic functions. In order to understand this class of functions in the abstract metric setting one needs to investigate their properties in the classical setting of Euclidean domains, or in the wider class of Riemannian manifolds.

Recall, that by a \emph{metric measure space} we denote metric space $(X,d)$ equipped with Borel regular measure $\mu$, which assigns to every ball a positive and finite value. In this setting we introduce the following class of functions.
\begin{definition}[Definition 3.1 in \cite{H}]\label{sh}
Let $\Omega \subset X$ be an open set. We call any locally integrable function $u: \Omega \rightarrow \mathbb{R}$ a \emph{strongly harmonic} in $\Omega$ if for all balls $B(x,r) \Subset \Omega$ there holds
\[ u(x) = \vint{B(x,r)} u(y) d\mu(y).\]
We call a radius $r>0$ \emph{admissible} at some $x \in \Omega$ whenever $B(x,r) \Subset \Omega$.  The space of all strongly harmonic functions on $\Omega$ is denoted by $\Hs(\Omega,d,\mu)$.
\end{definition}

The main subject of this work is a characterization of strongly harmonic functions on a certain class of metric measure spaces. Namely, we consider an open subset $\Omega \subset \Real^n$ equipped with a norm induced metric $d$ and a weighted Lebesgue measure $d\mu = w(x) dx $, $w \in L^1_{loc}(\Omega)$, $w>0$ a.e.

Bose, Flatto, Friedman, Littman, Zalcman studied the mean value property in the Euclidean setting, see \cite{B2,B3,B4,F2,F3,F1,F5,F}. We extended their results with our main result, see Theorem \ref{twsys} below. It generalizes Theorem 1 in \cite{F5} (see Theorem \ref{F-L} below) and Theorem 1 in \cite{B4} (see Theorem \ref{twbo} below) in the following ways:
\begin{enumerate}[{(1)}]
\item we consider general metric functions induced by a norm, not necessarily the Euclidean one,
\item we allow more general measures, i.e. the weighted Lebesgue measures $d\mu = w dx$,
\item the Bose approach uses derivatives of strongly differentiable weights and techniques suitable only for that case, whereas we only need to assume Sobolev regularity of $w$.
\end{enumerate}

Throughout the paper we use the multi-index notation proposed e.g. in the Evans' book \cite{E2}.
\begin{theorem}\label{twsys}
Let $\Omega \subset \Real^n$ be an open set. Let further $(\Omega,d,\mu)$ be a metric measure space equipped with a norm induced metric $d$ and a weighted Lebesgue measure $d\mu=wdx$, $w \in L^1_{loc}(\Omega)$, $w>0$ a.e. Suppose that the weight $w \in W^{2l,\infty}_{loc}(\Omega)$ for some given $l \in \Natural_+$. Then every function $u \in \Hs(\Omega,d,w dx)$ is a weak solution to the following system of partial differential equations
\begin{equation}\label{system3}
 \sum_{|\alpha|=j} A_\alpha \left( D^\alpha (uw) - uD^\alpha w \right) =0, \qquad \text{for }j=2,4,\ldots,2l.
\end{equation}
Coefficients $A_\alpha$ are defined as follows:
\[ A_\alpha:= \binom{\abs{\alpha}}{\alpha} \int_{B(0,1)} x^\alpha dx = \frac{\abs{\alpha}!}{\alpha_1! \cdot \ldots \cdot \alpha_n!} \int_{B(0,1)} x_1^{\alpha_1}\cdot \ldots \cdot x_n^{\alpha_n} dx. \]
\end{theorem}

In order to prove Theorem \ref{twsys} we need to establish regularity results which are stated as Proposition \ref{1p} and Theorem \ref{reg}. Namely, we show that if weight $w$ is locally bounded and in the $W^{1,p}_{loc}$, then all strongly harmonic functions are in $W^{1,p}_{loc}$ and that if $w$ is in $W^{l,\infty}$, then all strongly harmonic functions are in $W^{l,\infty}_{loc}$.

Our second main result is the following converse to Theorem \ref{twsys}.
\begin{theorem}\label{twod}
Let $\Omega \subset \Real^n$ be an open set and $(\Omega,d,\mu)$ be a metric measure space equipped with a norm induced metric $d$ and a weighted measure $d\mu = w dx$. Suppose that weight $w$ is analytic and positive in $\Omega$. Then, any solution $u$ to system of equations \eqref{system3} is strongly harmonic in $\Omega$.
\end{theorem}

Another, perhaps most surprising results are presented in Section 4 where we illustrate Theorem \ref{twsys} with the following observations:

\medskip
\noindent
\textit{If $p\ne 2$ and $n=2$, then the space $\Hs(\Omega,l^p,dx)$ is spanned by 8 linearly independent harmonic polynomials.}
\mbox{}

We already know that $\Hs(\Omega,l^2,dx)$ consists of all harmonic functions in $\Omega$, which differs significantly from the previous case. It is worthy mentioning here, that this dimension coincides with the number of linear isometries of $(\Real^n, l^p)$, which is $2^n n!$ and is computed in \cite{LI}. For more information see Section 4.2, 4.3 and 5.1.

\newpage
\noindent
\textbf{The organization of the paper.}

In Preliminaries we introduce basic notions and definitions, which will be essential in further parts of the paper. Among them we give the definitions of strongly and weakly harmonic functions on metric measure spaces and motivations for their formulations.

In Section 3 we present a historical sketch of the studies of the mean value property ending with the proof of Theorem \ref{twsys}. Moreover, by assuming the Sobolev regularity of weights, we prove in Theorem \ref{reg} that strongly harmonic functions are in the Sobolev space of the same order as the weight, see also Proposition~\ref{1p}. Further on we recall results of Flatto and Friedman--Littman concerning functions with the mean value property in the sense of Flatto (see \eqref{fmean} below) and compare them to strongly harmonic functions. Then, we recall a result of Friedman--Littman \cite{F5} which characterizes functions with the mean value property in the sense of Flato for the Lebesgue measure, but a metric not necessarily the Euclidean one. In fact we extend their proof to describe such functions via a system of PDEs. On the other hand, we present another approach studied by Bose \cite{B2,B3,B4}. He considered a mean value property on Euclidean balls for a weighted Lebesgue measure. We generalize both approaches in Theorem \ref{twsys} to the case of a weighted Lebesgue measure and a norm induced metric. We show that this case is the only one in which strongly harmonic functions coincide with those having mean value property in the sense of Flatto.

In Section 4 we focus on $l^p$ metrics for $1\leq p \leq \infty$. Equations of system~\eqref{system3} are calculated explicitly with their coefficients $A_\alpha$. We show that there appear only two distinct cases: either $p=2$ and $\Hs(\Omega,l^2,dx)$ consists of all functions which Laplacian vanishes in $\Omega$, or $p\ne 2$ and there are only finitely many linearly independent strongly harmonic functions in the space $\Hs(\Omega,l^p,dx)$. Similar observation can be obtained in higher dimensions using our techniques.

The last Section is devoted to proving Theorem \ref{twod}, a converse to Theorem \ref{twsys}. In order to complete that goal we recall the notion of generalized Pizzetti formula following Zalcman \cite{F}. Moreover, in Lemma \ref{lem4} we prove that equation for $j=2$ of \eqref{system3} is of the elliptic type. We use this fact to prove a regularity of strongly harmonic functions, i.e. that all strongly harmonic functions are analytic whenever weight is analytic.

We conclude the Section 5 with applying Theorem \ref{twod} to obtain the following peculiar observation. Suppose that $u$ is strongly harmonic, weight $w$ is smooth and metric is Euclidean. Then, $u$ is strongly harmonic with respect to infinitely many weights obtained as compositions of the Laplacian on $w$, \mbox{i.e. $\Delta^l w$ for $l \in \Natural$.}

\medskip
\noindent
\textbf{Acknowledgements:} A.K. is particularly grateful to his academic advisor Tomasz Adamowicz for important comments, fruitful discussions and valuable lessons in many cases extending results of the paper.
\section{Preliminaries}

In this section we introduce basic notions used in further parts of the work.


Let $A \subset \mathbb{R}^n$ be any set of positive Lebesgue measure, $|A|\, >0$, and a measurable function $f: \mathbb{R}^n\rightarrow \mathbb{R}$, then the \emph{mean value of $f$} over set $A$ will be denoted by
\begin{equation*}
\vint{A} f(x) dx := \frac{1}{|A|}\int\limits_{A} f(x) dx.
\end{equation*}

We say that a function $u \in C^2(\Omega)$ is \emph{harmonic} in an open set $\Omega \subset \mathbb{R}^n$, if  $ \Delta u =0$ in $\Omega$. One of several properties of harmonic functions is the Gauss theorem stating that if $u$ is harmonic, then it has mean value property with respect to balls and spheres. There is an elegant converse relation between the mean value property and harmonicity brought by Hansen-Nadirashvili in \cite{A}: \textit{Let $\Omega$ be an open bounded subset of $\mathbb{R}^n$, $u \in C( \Omega) \cap L^\infty(\Omega)$ be such that for every $x \in \Omega$ there exists $0< r^x \leq \dist(x,\partial \Omega)$ with the property $ u(x) = \vint{B(x,r^x)} u(y) dy$. Then $u$ is harmonic in $\Omega$.}

The aforementioned relation between harmonicity and the mean value property leads to formulating a relaxed version of the strong harmonicity (cf. Definition \ref{sh}): \textit{Let $\Omega \subset X$ be an open set. We call any locally integrable function $u: \Omega \rightarrow \mathbb{R}$ \emph{weakly harmonic} in $\Omega$ if for all points $x \in \Omega$ there exists at least one radius $0 < r^x < \dist(x, \partial \Omega)$ with the following property $u(x) = \vint{B(x,r^x)} u(y) d\mu(y)$.} For further information about properties of weakly and strongly harmonic functions we refer to \cite{H,J}.

Let us consider a function $f: \Real^n \rightarrow \Real$. For $x,h \in \Real^n$ we define the \emph{difference} of $f$ at $x$ as follows
\[ \Delta_h f(x) := f(x+h) -f(x).\] 
Observe, that for any $h,h' \in \Real^n$ difference operators $\Delta_h$ and $\Delta_{h'}$ commute. In what follows we use the multi-1index notation proposed e.g. in the Evans' book \cite{E2}. For $h=(h_1,\ldots,h_n )$, $h_i \in \Real^n \setminus \{0\}$ and $\alpha=(\alpha_1,\ldots,\alpha_n) \in \Natural^n$ the \emph{$\alpha$-th difference} is defined as follows
\begin{equation}\label{dif}
 \Delta_h^\alpha f(x) := (\Delta_{h_1})^{\alpha_1} \circ (\Delta_{h_2})^{\alpha_2} \circ \ldots \circ (\Delta_{h_n})^{\alpha_n} f(x) .
\end{equation}
The \emph{$\alpha$-th difference quotient} of $f$ is the following expression $\frac{\Delta_h^\alpha f(x)}{|h|^\alpha} : =\frac{\Delta_h^\alpha f(x)}{|h_1|^{\alpha_1} \ldots |h_n|^{\alpha_n}}$. Formulas describing the difference of a multiple and a quotient of two functions are similar to formulas describing their derivatives. Let us consider two functions $f,g: \Real^n \rightarrow \Real$ with $g>0$. In what follows we will need a representation of the $\alpha$-th difference quotient of $f/g$ in terms of difference operators applied to the nominator $f$ and the denominator $g$. Observe, that for $\alpha \in \Natural^n$ there holds a difference quotients analogue of the Leibniz formula
\begin{equation}\label{ee0}
\Delta^\alpha_h \left( \frac{f}{g} \right) (x)  = \sum_{\beta \leq \alpha} \binom{\alpha}{\beta} \Delta_h^{\alpha - \beta}  f(x + \beta h) \Delta_h^\beta \left( \frac{1}{g} \right)(x),
\end{equation} 
where by $\beta \leq \alpha$ we understand that $\beta_i \leq \alpha_i$ for all $i=1,\ldots,n$. Using \eqref{ee0} we only need to express the $\beta$-th difference quotient of $1/g$ in terms of difference quotients of $g$. To reach that goal we use a discrete variant of the Fa\'a di Bruno formula developed in \cite{FAA}, from which one can derive the following result.
\begin{proposition}[Theorem 1.3 in \cite{FAA}]\label{pq}
Let $\beta \in \Natural^n$, $x \in \Real^n$, $h =(h_1,\ldots,h_n )$, $h_i \in \Real^n  \setminus \{0\}$ and $g:\Real^n \rightarrow \Real$ be positive. Then
\begin{equation}\label{32}
 \frac{\Delta^\beta_h \left( \frac{1}{g}\right) (x)}{|h|^\beta} = \sum_{\scriptscriptstyle \beta^1 + \ldots + \beta^m = \beta} \frac{(-1)^m m!}{g^{m+1}(x)} \frac{\abs{\Delta_h^{\beta^1} g(x)} \ldots \abs{ \Delta_h^{\beta^m} g(x)} }{\abs{h}^\beta} + Err(h),
\end{equation}
where we sum with respect to $\beta^i \in \Natural^n \setminus \{0\}$ for $i=1,\ldots,m$, $m\in \Natural_+$ and the error term satisfies $Err(h) \rightarrow 0$ with $\abs{h}^\beta \rightarrow 0$.
\end{proposition}

An outcome of the above discussion is that we can represent the $\alpha$-th difference quotient of $f/g$ as the sum of fractions whose numerators, apart from constants, consist only of terms $\Delta^{\beta-\alpha}_h f(x+ \beta h)$, $ \Delta_h^{\beta^j} g(x) $ and their products for some $\beta^1 + \ldots + \beta^m = \beta \leq \alpha$. Furthermore, the operator $\Delta_h$ appears in each of these numerators exactly $\abs{\alpha}$-times, which can be justified by calculating the sum $\sum_{j=1}^k \abs{\beta_j} + \abs{\alpha-\beta}= \abs{\alpha} $.

We use difference quotients to prove regularity of strongly harmonic functions in Theorem \ref{reg}. Therefore, we gather below a characterization of Sobolev functions via their difference quotients.
\begin{theorem}[Theorem 3, p. 277 in \cite{E2}]\label{w1p}
Let $\Omega \subset \Real^n$ be an open set and $f \in L^p_{loc}(\Omega)$. Then $f \in W^{1,p}_{loc}(\Omega)$ if and only if for every $K \Subset \Omega$ there exists $C_K>0$ such that $\norm{\frac{\Delta_h f}{\abs{h}}}_{L^p(K)} \leq C_K$ holds for all $h$ such that $2 \abs{h} < \dist (K,\partial \Omega)$.
\end{theorem}
Moreover, in case $p=\infty$ we derive from Theorem 5 in \cite{EG} the following result.
\begin{proposition}\label{wkp}
Let $\Omega \subset \Real^n$ be an open set, $f:\Omega \rightarrow \Real$ and $k \in \Natural_+$. Then $f \in W^{k,\infty}_{loc}(\Omega)$ if and only if for all $K \Subset \Omega$ and all multi-indices $\alpha \in \Natural^n$ such that $\abs{\alpha} \leq k $ there exists $C_{K,\alpha} >0$ such that for all $t \in \Real^n$ with $t_i \ne 0$ and $2 \abs{\alpha_1 t_1 + \ldots + \alpha_n t_n} < \dist(K,\partial \Omega)$ there holds 
\[\norm{\frac{\Delta^\alpha_h f }{\abs{t}^\alpha} }_{L^\infty_{loc} (\Omega)} \leq C_{K,\alpha}, \]
where $\Delta^\alpha_h$ is defined in \eqref{dif}, $h = et=(t_1 e_1, t_2 e_2,\ldots,t_n e_n)$ and $e=(e_1,\ldots,e_n)$ is the standard basis of $\Real^n$.
\end{proposition}
The proof of this result follows the same steps as the one in \cite{EG} and therefore we will omit it.
\section{Strongly harmonic functions on open subsets of $\Real^n$}

In this section we focus our attention on the class of strongly harmonic functions appearing in Definition~\ref{sh}. Let $(X,d,\mu)$ be a metric measure space with a Borel measure $\mu$. We denote by $\Hs (\Omega, d,\mu)$ the set of all strongly harmonic functions on an open domain $\Omega \subset X$. In what follows we will omit writing the set, metric or measure whenever they are clear from the context.

The key results of this section are Proposition \ref{1p} and Theorem \ref{reg}. There, we show the Sobolev regularity for functions in $\Hs (\Omega,d,\mu)$ for the weighted Lebesgue measure $d \mu = w dx$ depending on the Sobolev regularity of weight $w$. The properties of strongly and weakly harmonic functions were broadly studied in \cite{J,H} and in \cite{H1} in the setting of Carnot groups. Below, we list out some of those properties especially important for further considerations.

\begin{proposition}[Proposition 4.1 in \cite{J}]\label{prop41}
Suppose that measure $\mu$ is \emph{continuous with respect to metric}~$d$, i.e. for all $r>0$ and $x \in X$ there holds $ \lim_{d(x,y) \rightarrow 0} \mu\left( B(x,r)\triangle B(y,r) \right)=0,$ where we denote by \mbox{$E \triangle F := (E \setminus F ) \cup (E \setminus F)$} the symmetric difference of $E$ and $F$. Then $\Hs(\Omega,d,\mu)~\subset~C(\Omega)$.
\end{proposition}
Moreover, the Harnack inequality and the strong maximum principle hold for strongly harmonic functions as well as the local H\"older continuity and even local Lipschitz continuity under more involved assumptions, see \cite{H}. It is important to mention here that similar type of problems were studied for a more general, nonlinear mean value property by Manfredi--Parvainen--Rossi and Arroyo--Llorente, see~\cite{man,LL1,LL2,LL3}.

We know that $\Hs$ is a linear space, but verifying by using the definition whether some function satisfies the mean value property might be a complicated computational challenge. From that comes the need for finding a handy characterization of class $\Hs$, or some necessary and sufficient conditions for being strongly harmonic.

Our goal is to characterize class $\Hs$ if $X = \Real^n$ equipped with a distance $d$ induced by a norm and a weighted Lebesgue measure $d\mu=w(x) dx$. 
\begin{center}
\textit{From now on we a priori assume that a function $w \in L_{loc}^1(\Omega)$ and $w>0$ almost everywhere in $\Omega$.}
\end{center}
Let us begin with noting that strongly harmonic functions in such setting are continuous.

\begin{proposition}
Let $\Omega \subset \Real^n$ be an open set. Then $\Hs(\Omega,d,w(x) dx) \subset C(\Omega).$
\end{proposition}
\begin{proof}
Observe that $\mu(\partial B(x,r))= \int_{\partial B(x,r) } w(y) dy= 0$. The proof follows then by Lemma 2.1 from~\cite{J} and Proposition \ref{prop41}.\end{proof}

Let us observe that the proof of continuity of strongly harmonic functions works for all weights $w$. However, in order to show existence and integrability of weak derivatives we need to assume Sobolev regularity of $w$.

\begin{proposition}\label{1p}
Let $\Omega \subset \Real^n$ be an open set, $d$ be a norm induced metric and a weight $w \in W^{1,p}_{loc}(\Omega) \cap L^\infty_{loc}(\Omega)$ for some $1<p<\infty$. Then $ \Hs(\Omega,d,w ) \subset W^{1,p}_{loc}(\Omega).$
\end{proposition}

\begin{proof}
Fix a compact set $K \Subset \Omega$. Moreover, let $r=\frac{1}{4} \dist(K,\partial \Omega)$. Fix $h \in \Real^n$ with $|h| < r$. Denote by $K':= \{z \in \Omega: \dist(z,K)\leq 2r \}$. Let us observe that due to the first assertion of Lemma 2.1 in \cite{H} by continuity of the measure $\mu = w dx$ with respect to the metric $d$ there exist $0<m := \inf_{x \in K'} \mu(B(x,r))$. The difference quotient of $u$ at $x \in K$ reads
\begin{equation*}
\abs{\Delta_h u(x)} = \abs{u(x+h) - u(x) } = \abs{ \frac{\int_{B(x+h,r)} uw}{\int_{B(x+h,r)} w} -  \frac{\int_{B(x,r)} uw}{\int_{B(x,r)} w} },
\end{equation*}
where we used the mean value property. Now we add and subtract a term $\frac{\int_{B(x,r)} uw }{\int_{B(x+h,r)} w}$ and use the triangle inequality to get
\begin{equation}\label{ew01}
\abs{\Delta_h u(x)} \leq  \abs{\frac{\int_{B(x+h,r)} uw}{\int_{B(x+h,r)} w} -  \frac{\int_{B(x,r)} uw }{\int_{B(x+h,r)} w}} + \abs{\frac{\int_{B(x,r)} uw }{\int_{B(x+h,r)} w} - \frac{\int_{B(x,r)} uw}{\int_{B(x,r)} w} } .
\end{equation}
The first term can be estimated as follows
\begin{align}
\abs{\frac{\int_{B(x+h,r)} uw - \int_{B(x,r)} uw }{\int_{B(x+h,r)} w}} & \leq \frac{1}{\int_{B(x+h,r)} w} \abs{ \int_{B(x+h,r)} uw - \int_{B(x,r)} uw}  \leq  \frac{1}{m} \int\limits_{B(x+h,r) \triangle B(x,r) } \abs{uw} \nonumber \\ 
& \leq  \frac{\norm{uw}_{L^\infty(K')}}{m} \abs{ B(x+h,r) \triangle B(x,r)}, \label{33}
\end{align}
where the last term denotes the $n$-dimensional Lebesgue measure of the symmetric difference of two balls. To manage this term we refer to Theorem 3 in \cite{DS} to get that
\begin{equation}\label{35}
\abs{ B(x+h,r) \triangle B(x,r)} \leq \abs{h} \abs{\partial B(x,r)} = |h| c_n r^{n-1},
\end{equation}
where in the last term $c_n$ stands for the $(n-1)$-dimensional Lebesgue measure of the unit sphere.

The second term of \eqref{ew01} reads
\begin{align}
\abs{ \frac{\int_{B(x,r)} uw }{\int_{B(x+h,r)}w} - \frac{\int_{B(x,r)} uw}{\int_{B(x,r)} w}} & \leq \frac{\int_{B(x,r)} \abs{uw} }{ \int_{B(x+h,r)} w \int_{B(x,r)} w} \abs{\int_{B(x+h,r)} w(y) dy - \int_{B(x,r)} w(y) dy } \nonumber \\
& \leq \frac{\norm{uw}_{L^\infty(K')} }{m^2} \abs{\int_{B(x,r)} (w(y+h)-w(y)) dy } \nonumber \\ 
&\leq \frac{\norm{uw }_{L^\infty(K')} }{m^2} \int_{B(x,r)} \abs{\Delta_h w(y) } dy. \label{34}
\end{align}
By gathering together both terms of \eqref{ew01} we obtain the following
\begin{align*}
\int\limits_K \left( \frac{\abs{\Delta_h u(x)}}{|h|} \right)^p dx & \leq 2^{p-1}\norm{uw}^p_{L^\infty(K')} \int\limits_K  \Bigg\lbrack \frac{ c_n^{p} r^{p(n-1)}}{m^p}  + \frac{ 1 }{m^{2p}}  \left( \int_{B(x,r)} \frac{\abs{\Delta_h w(y) }}{|h|} dy \right)^p \Bigg\rbrack  dx.
\end{align*}
The first term is bounded, therefore we only need to take care of the second one. For the sake of simplicity we omit writing the constant $2^{p-1} m^{-2p}    \norm{uw}^p_{L^\infty(K')}$ and use the Jensen inequality
\[\int\limits_K \left( \int\limits_{B(x,r)}  \frac{\abs{\Delta_h w(y) }}{|h|} \right)^p dy  \, dx  \leq \abs{B(0,1)}^p r^{np} \int\limits_K \int\limits_{B(x,r)}  \left( \frac{\abs{\Delta_h w(y) }}{|h|} \right)^p  dy \, dx < \infty.\]
This integral is finite by the assumptions on regularity of $w$ and Theorem \ref{w1p} applied to weight $w$ with an observation, that from the proof of Theorem \ref{w1p} follows the uniform boundedness of constants $C_{\overline{B(x,r)}}$ with respect to $x \in K$. Therefore, by applying this theorem again to $u$ we obtain that $u \in W^{1,p}_{loc}(\Omega)$.
\end{proof}

We prove higher regularity of strongly harmonic functions by using difference quotients characterization of the space $W^{k,\infty}_{loc}$ presented in Proposition \ref{wkp}.

\begin{theorem}\label{reg}
Let $\Omega \subset \Real^n$ be an open set, $d$ be a norm induced metric and a weight $w \in W^{l,\infty}_{loc}(\Omega)$ for some $l \in \Natural_+ $. Then $ \Hs(\Omega,d,w ) \subset W^{l,\infty}_{loc}(\Omega).$
\end{theorem}

\begin{proof}
Let $u \in \Hs(\Omega,d,w)$ and $w$ be as in the assumptions. We will show the assertion using the mathematical induction with respect to $k \leq l$ proving that $u \in W^{k,\infty}_{loc}(\Omega)$. 

Let $k=1$ and $K,K',r,h,m$ be as in the proof of Proposition \ref{1p}. The following is the consequence of \eqref{ew01}, \eqref{33} and \eqref{34}
\begin{equation*}
\frac{\abs{\Delta_h u(x)}}{\abs{h}} \leq \norm{uw}_{_{L^\infty (K')}} \left( \frac{c_n r^{n-1}}{m} + \frac{\int_{B(x,r)} \abs{\Delta_h w}}{\abs{h} m^2} \right) \leq \norm{uw}_{_{L^\infty (K')}} \left( \frac{c_n r^{n-1}}{m} + \frac{\abs{B(0,1)}r^n \norm{ \frac{\Delta_h w}{\abs{h}}}_{L^\infty(K')}  }{m^2} \right)
\end{equation*}
and is bounded by Proposition \ref{wkp}.

Now let us assume that $u \in W^{k-1,\infty}_{loc} (\Omega)$ and choose $K$, $r$ as above, \mbox{$K':= \{z \in \Omega: \dist(z,K)\leq 2r \}$}. We consider the $\alpha$-th order difference quotient of $u$ for $\abs{\alpha} =k$. Let $t\in \Real^n$ be such that \mbox{$|\alpha_1t_1 + \ldots + \alpha_n t_n|< \frac{r}{2k}$} and define $h = (t_1 e_1, \ldots, t_n e_n)$. Proposition \ref{pq}, formula \eqref{ee0} and the related discussion applied to \sloppy ${f(x) = \int_{B(x,r)} uw}$ and $g(x) = \int_{B(x,r)} w$ allows us to reduce the discussion to showing that 
\[ \frac{ \left( \Delta_h^{\alpha - \beta} \int_{B(x+\beta h,r)} uw \right) \, \prod\limits_{i=0}^m \left( \Delta_h^{\beta^i} \int_{B(x,r)} w \right)}{|t|^\alpha} \] 
is bounded for any $\beta^1 + \ldots + \beta^m = \beta \leq \alpha$, cf. \eqref{32}. Let us observe that \[ \frac{\Delta_h^{\beta^i} \int_{B(x,r)} w }{\abs{t}^{\beta^i}} = \int_{B(x,r)} \frac{\Delta_h^{\beta^i} w }{\abs{t}^{\beta^i}} \] is bounded by the regularity of $w$ and Proposition \ref{wkp}. Therefore, we only need to show boundedness of the term $\abs{t}^{\beta -\alpha}\Delta_h^{\alpha - \beta} \int_{B(x+\beta h,r)} uw $. We deal with the case $\beta = 0$, since for $\abs{\beta}>0$ the difference quotient of order $\abs{\alpha - \beta} \leq k-1$ is bounded due to $uw \in W^{k-1,\infty}$. Observe that there exists $j$ such that $\alpha_j \ne 0$ and $\Delta_h^\alpha =\Delta_{t_je_j} \Delta_h^{\alpha - e_j}$, hence
\begin{align*}
\frac{1}{\abs{h}^\alpha}\abs{ \Delta^\alpha_h \int_{B(x,r)} uw}  & = \frac{1}{\abs{h}^\alpha} \abs{ \Delta_{t_j e_j} \int_{B(x,r)} \Delta^{\alpha - e_j}_h uw(y) dy} \\ & = \frac{1}{\abs{h}^\alpha} \abs{\int_{B(x+t_j e_j,r)} \Delta^{\alpha-e_j}_h uw(y) dy -\int_{B(x,r)} \Delta^{\alpha-e_j}_h uw(y) dy }\\
& \leq \frac{1}{\abs{t_j}} \int_{B(x+t_j e_j,r) \triangle B(x,r) } \frac{\abs{\Delta^{\alpha-e_j}_h uw(y) }}{\abs{h}^{\alpha-e_j}}dy \leq c_n r^{n-1}  \norm{\frac{ \Delta^{\alpha-e_j}_h uw }{\abs{h}^{\alpha-e_j}}}_{L^\infty(K)},
\end{align*}
which is bounded on behalf of the regularity assumption on both $u$ and $w$ (in the last estimate we have also used \eqref{35}). Therefore, we conclude that $u\in W^{l,\infty}_{loc}(\Omega)$, which ends the proof.
\end{proof}

\subsection{Proof of Theorem \ref{twsys}}

The consequences of the mean value property has been studied in the 1960's by Flatto \cite{F2,F3}, Friedman--Litmann \cite{F5} and Bose \cite{B2,B3,B4}. Subsequently, their work was extended by Zalcman \cite{F}. Below, we briefly discuss these results. According to our best knowledge, the investigation in this area originate from a work by Flatto \cite{F2}. He considered functions with the following property:

Let us fix an open set $\Omega \subset \Real^n$ and a bounded set $K \subset \Real^n$. Moreover, let $\mu$ be a probabilistic measure on $K$ such that all continuous functions on $K$ are $\mu$-measurable and for all hyperplanes $V \subset R^n$ it holds that $ \mu(K \cap V) <1$, i.e. $\mu$ is not concentrated on a hyperplane. We will say that a continuous function $u \in C(\Omega)$ has the mean value property in the sense of Flatto, if
\begin{equation}\label{fmean}
u(x) = \int\limits_K u(x+ry) d\mu(y)
\end{equation}
for all $x \in \Omega$ and $0 < r < r(x) $. Let us observe that for $K = B(0,1)$ and $\mu$ being the normalized Lebesgue measure on $K$, property \eqref{fmean} is equivalent to the strong harmonicity of $u$ in $\Omega$ by the following formula
\begin{equation}\label{obs}
u(x)= \vint{B(x,r)} u(z) dz = \vint{B(0,1)} u(x+r y ) dy = \int_K u(x+r y ) d\mu(y).
\end{equation}
This holds exactly for homogeneous and translation invariant metrics, because only then 
\[B(x,r)=x+r \cdot B(0,1) , \]
where we understand the above equality in the sense of the Minkowski's addition and multiplication, see~\cite{OO}. For such distance functions one can obtain any ball $B(x,r)$ from $B(0,1)$ by using the change of variables $ y=\frac{z-x}{r}.$ In relation to homogeneous and translation invariant distance let us recall the following lemma, which is likely a part of the mathematical folklore. However, in what follows we will not appeal to this observation.
\begin{lemma}
If $d$ is a translation invariant and homogeneous metric on $\Real^n$, then there exists a norm $\| \cdot \|$ on $\Real^n$ such that for all $x,y \in \Omega$ there holds that $d(x,y) = \| x-y \| $. 
\end{lemma}
%
%
%
%

We present also a characterization of all such metrics on $\Real^n$ by using the Minkowski functional, see~\cite{OO}. From now on we use the word \emph{symmetric} to describe a set $K\subset \Real^n$ which is symmetric with respect to the origin of the coordinate system. For any nonempty convex set $K$ we consider the Minkowski functional.
\begin{lemma}[p.54 in \cite{OO}]\label{mink}
Suppose that $K$ is a symmetric convex bounded subset of $\Real^n$, containing the origin as an interior point. Then, its Minkowski functional $n_K$ defines a norm on $\Real^n.$ Moreover, if $\|\cdot \|$ is a norm on $\Real^n$, then the Minkowski functional $n_K$, where $K$ is a unit ball with respect to $\| \cdot \|$, is equal to that norm.
\end{lemma}
\begin{example}\label{uwmink}
Among many examples of norm induced metrics on $\Real^n$ are $l^p$ distances for $1\leq p \leq \infty$. Moreover, let us fix numbers $a_i>0$ for $i=1,\ldots,n$, set $a:=(a_1,\ldots,a_n)$ and $1\leq p< \infty$ and define
\[ \| x\|_p^a := \left( \sum_{i=1}^p \left( \frac{|x_i|}{a_i} \right)^p \right)^{\frac{1}{p}} .\]
In case $p=2$ all balls with respect to $\| \cdot \|_p^a$ are ellipsoids with the length of semi-axes equal to $a_i$ in $x_i$'s axes direction respectively.

Let us observe that by Lemma \ref{mink} there is the injective correspondence between norms on $\Real^n$ and a class of all symmetric convex open bounded subsets $K$ of $\Real^n$. More specifically, every $K$ defines a norm on $\Real^n$ through the Minkowski functional and vice versa, given a norm on $\Real^n$ the unit ball $B(0,1)$ is a symmetric convex open bounded set, therefore provides an example of $K$. This can be expressed in one more way, namely that all norms can be distinguished by their unit balls, so to construct a norm we only need to say what is its unit ball. Therefore, further examples of norms can be constructed for any $n$-dimensional symmetric convex polyhedron $K$. All balls with respect to $n_K$ will be translated and dilated copies of $K$.
\end{example}

The formula \eqref{obs} is true only if the measure of a ball scales with the power $n$ of its radius, the same which appears in the Jacobian from the change of variables formula. This is true only for measures which are constant multiples of the Lebesgue measure. Note that \eqref{fmean} does not coincide, in general, with the one presented in our work, since the Flatto's mean value is calculated always with respect to the same fixed reference set $K$ and measure $\mu$, whose support is being shifted and scaled over $\Omega$. Whereas, in Definition \ref{sh} the measure is defined on the whole space, and as $x$ and $r$ vary, the mean value is calculated with respect to different weighted measures. Indeed, from the point of view of Flatto, the condition from Definition \ref{sh} reads
\[ u(x) = \int_{X} u(y) \frac{d \mu |_{B(x,r)}}{\mu(B(x,r))}. \]
This mean value property cannot be written as an integral with respect to one fixed measure for different pairs of $x$ and $r$, even when \eqref{obs} holds.

Flatto discovered that functions satisfying \eqref{fmean} are solutions to a second order elliptic equation, see \cite{F2}. However, from the point of view of our discussion, more relevant is the following later result.

\begin{theorem}[Friedman--Littman, Theorem 1 in \cite{F5}]\label{F-L}
Suppose that $u$ has property \eqref{fmean} in $\Omega \subset \Real^n$. Then $u$ is analytic in $\Omega$ and satisfies the following system of partial differential equations
\begin{equation}\label{system}
 \sum_{|\alpha|=j} A_\alpha D^\alpha u =0 \qquad \text{for } j=1,2,\ldots
\end{equation}
The coefficients $A_\alpha$ are moments of measure $\mu$ and are defined by $A_\alpha := \binom{\abs{\alpha}}{\alpha} \int_K x^\alpha d\mu(x)$. Moreover, any function $u \in C^\infty (\Omega)$ solving system \eqref{system} is analytic and has property \eqref{fmean}.
\end{theorem}

\begin{remark}
Theorem \ref{F-L} gives full characterization of $\Hs(\Omega,d)$ for $d$ being induced by a norm. Theorem 3.1 in \cite{F2} states that all functions having property \eqref{fmean} are harmonic with respect to variables obtained from $x$ by using an orthogonal transformation and dilations along the axes of the coordinate system. On the other hand the proof of Theorem 1 in \cite{F5} shows that the equation in system \eqref{system} corresponding to $j=2$ is always elliptic with constant coefficients from which the analyticity follows.
\end{remark}

Flatto as well as Friedman and Littman described in their works the space of functions possessing property \eqref{fmean}. We present appropriate results below.

\begin{proposition}[Friedman--Littman, Theorem 2 in \cite{F5}]\label{FL2}
The space of solutions to system \eqref{system} is finitely dimensional if and only if the system of algebraic equations $\sum_{|\alpha|=j} A_\alpha z^\alpha = 0$ for $j=1,2,\ldots$ has the unique solution $z=(z_1,\ldots,z_n)=0$, where $z_i \in \mathbb{C}$.
\end{proposition}

\begin{remark}\label{uw6}
From the proof of Proposition \ref{FL2} it follows that if there exists a nonpolynomial solution to~\eqref{system}, then the solution space is infinitely dimensional. If the dimension finite, then all strongly harmonic functions are polynomials.
\end{remark}

A rather different approach to the mean value property and its consequences was studied by Bose, see \cite{B2,B3,B4}. He considered strongly harmonic functions on $\Omega \subset \Real^n$ equipped with non-negatively weighted measure $\mu = w(x) dx$, for a weight $w \in L_{loc}^1(\Omega)$ being a.e. positive in $\Omega$ and only a metric $d$ induced by the $l_2$-norm. Under the higher regularity assumption of weight $w$, Bose proved the following result.
\begin{theorem}[Bose, Theorem 1 in \cite{B4}]\label{twbo}
If there exists $l \in \Natural$ such that $w \in C^{2l+1}(\Omega)$, then every $u \in \Hs(\Omega,w)$ solves the following system of partial differential equations
\begin{equation}\label{system2}
\Delta u \Delta^j w  + 2 \nabla u \nabla \left( \Delta^j w \right)=0, \qquad \text{for } j=0,1,\ldots,l,
\end{equation}
where $\Delta^j$ stands for the $j$th composition of the Laplace operator $\Delta$. If $w$ is smooth, then equations \eqref{system2} hold true for all $j \in \Natural$.
\end{theorem}

The converse is not true for smooth weights in general, see counterexamples on p. 479 in \cite{B2}. Furthermore, Bose proved in \cite{B4} the following result, by imposing further assumptions on $w$.

\begin{proposition}[Bose]\label{bosod}
Let $l \in \Natural$ and $w \in C^{2l}(\Omega)$. Suppose that there exist $a_0,\ldots,a_{l-1} \in \Real$ such that
\[ \Delta^l w = a_0 w + a_1 \Delta w + \ldots + a_{l-1} \Delta^{l-1} w.\]
Then any solution $u$ to \eqref{system2} for $j=0,1,\ldots,l-1$ is strongly harmonic, that is $u \in \Hs(\Omega,w)$.
\end{proposition}

The following result by Bose contributes to the studies of the dimension of space $\Hs(\Omega,w)$.
\begin{proposition}[Bose, Corollary 2 in \cite{B2}]
Suppose that $\Omega \subset \Real^n$ for $n>1$ and there exists $\lambda \in \Real$ such that $ \Delta w = \lambda w.$ Then $\text{dim} \, \Hs(\Omega,w) = \infty$.
\end{proposition}

\begin{remark}\label{rem1}
Before we present the proof of Theorem \ref{twsys} let us discuss the equations of system \eqref{system3}. First of all, by Remark \ref{uwmink} we know that the set $B(0,1)$ is symmetric with respect to the origin. If $|\alpha|$ is an odd number, then $x^\alpha$ is an odd function, hence $A_\alpha = 0$. Therefore only evenly indexed equations of \eqref{system3} are nontrivial, although we will prove them for all $j \leq 2l$. In fact, the proof of Theorem \ref{twsys} can be applied to functions with the mean value property over any compact set $K \subset \Real^n$, which does not necessarily need to be a unit ball with respect to a norm on $\Real^n$, i.e. to functions with the following property
\[ u(x) = \frac{1}{\int_K w(x+ry) dy}\int_{K} u(x+ry)w(x+ry) dy,\]
which holds for all $x \in \Omega$ and $0<r<r(x)$. In that case we obtain also oddly indexed equations of system~\eqref{system3} are nontrivial. 

If the unit ball is symmetric with respect to all coordinate axes, the coefficient $A_\alpha$ is zero whenever some $\alpha_i$ is odd. Therefore, in the $j$-th equation of \eqref{system3} occur only differential operators acting evenly on each of variables. Examples of norms for which $B(0,1)$ is symmetric with respect to all coordinate axes include the $l^p$ norms for $p \in [1,\infty]$, but also by Lemma \ref{mink} one can produce more examples.
\end{remark}


We are now in a position to present the proof of Theorem \ref{twsys}.

\begin{proof}[Proof of Theorem \ref{twsys}]
Let $u \in \Hs(\Omega,d,wdx)$ be as in assumptions of Theorem \ref{twsys}. Then, for $x \in \Omega$ and $0<r<\text{dist}(x,\partial \Omega)$ as in \eqref{obs} there holds
\[ u(x) \int\limits_{B(x,r)} w(y) dy = u(x) \int\limits_{B(0,1)} w(x+ry)r^n dy=\int\limits_{B(0,1)} u(x+ry)w(x+ry) r^n dy =\int\limits_{B(x,r)} u(y)w(y)dy.\]
Without the loss of generality we may assume that \[ B(0,1)=\{x : d(x,0) <1 \} \subset \{x: \|x\|_{2} \leq 1 \},\] since we can always restrict the set of admissible radii in the mean value property. The assertion is a local property, therefore we may restrict our considerations to the analysis of the behavior of $u$ on a ball $B \subset \Omega$ with $\text{dist}(B,\partial \Omega) =2 \varepsilon>0$. Furthermore, let $B'$ be a ball concetric with $B$ with $\varepsilon$ distance from $\partial \Omega$. We redefine $u$ and $w$ in the following way
\begin{equation*}
\widebar{u}(x)= u(x) \chi_{B'}(x) \qquad \widebar{w}(x) =  w(x) \chi_{B'}(x).
\end{equation*}

The function $\widebar{u}$ and the weight $\widebar{w}$ are both in $W^{2l,\infty}(B)$ since $B \Subset B'$. Let $\varphi \in C^\infty_0 (B)$. Then for all $x \in B$, $y \in B(0,1)$ and $0<r<\varepsilon$ it holds $u(x+ry)=\bar{u}(x+ry)$. Since $\varphi(x)=0$ outside of $B$ we have that for all $x \in \Real^n$ there holds
\begin{equation}\label{eq1}
\widebar{u}(x)\varphi(x)\int\limits_{B(0,1)} \widebar{w}(x+ry) =\varphi(x) \int\limits_{B(0,1)} \widebar{u}(x+ry)\widebar{w}(x+ry).
\end{equation}
For the sake of simplicity below we still use symbols $u$ and $w$ to denote $\widebar{u}$ and $\widebar{w}$, respectively. We integrate both sides of \eqref{eq1} with respect to $x\in \Real^n$ to obtain
\begin{equation}\label{eq2}
\int\limits_{\Real^n} u(x)\varphi(x) \bigg( \, \int\limits_{B(0,1)} w(x+ry) dy \bigg) dx= \int\limits_{\Real^n} \varphi(x) \bigg( \,\int\limits_{B(0,1)} u(x+ry)w(x+ry) dy \bigg) dx.
\end{equation}
Observe, that the Fourier transform of functions $\varphi u$, $\int_{B(0,1)} w(x+ry) dy$ and $\int_{B(0,1)} u(x+ry) w(x+ry) dy$ exist and the latter two are $L^2(\Real^n)$ integrable. Therefore, we apply the Parseval identity in \eqref{eq2} and obtain
\begin{equation}\label{eq3}
\int\limits_{\Real^n} \overline{\Fs ( \varphi u ) (\xi)} \Fs \bigg( \, \int\limits_{B(0,1)} w(\cdot+ry) dy \bigg) (\xi) d\xi = \int\limits_{\Real^n} \overline{\Fs(\varphi)(\xi)} \Fs \bigg( \, \int\limits_{B(0,1)} u(\cdot+ry)w(\cdot +ry) dy \bigg) (\xi) d\xi.
\end{equation}
Here $\Fs(f)(\xi):= \int_{\Real^n} e^{-i\xi y} f(y) dy$ stands for the Fourier transform of $f$ at $\xi \in \Real^n$. The following formula holds for any $f \in L^1_{loc} (\Omega)$:
\[ \int\limits_{\Real^n} e^{- i x \xi} \bigg( \, \int\limits_{B(0,1)} f(x+y)dy \bigg) dx = \int\limits_{B(0,1)} \Fs (f(\cdot+y) )(\xi) dy =\Fs(f)(\xi) \int\limits_{B(0,1)} e^{ i y \xi }  dy. \]
We apply this formula twice: for $f=w$ and $f=uw$ and employ respectively to the left- and the right-hand side in \eqref{eq3} to arrive at the following identity:
\begin{equation}\label{eq4}
\int\limits_{\Real^n} \overline{\Fs(\varphi u)(\xi)} \Fs(w)(\xi) \bigg( \, \int\limits_{B(0,1)}  e^{i ry \xi } dy \bigg) d\xi = \int\limits_{\Real^n} \overline{\Fs(\varphi)(\xi)} \Fs (uw)(\xi) \bigg(\, \int\limits_{B(0,1)} e^{ i ry\xi}dy \bigg) d\xi.
\end{equation}
Let us observe that both sides of \eqref{eq4} are smooth functions when considered with respect to $r$ and this allows us to calculate the appropriate derivatives by differentiating under the integral sign. Namely, we differentiate \eqref{eq4} with respect to $r$ by $j \leq 2l$ times
\[ \int\limits_{\Real^n} \overline{\Fs(\varphi u)(\xi)} \Fs(w)(\xi) \bigg( \, \int\limits_{B(0,1)}  ( i \xi y)^j e^{ i r y \xi } dy \bigg) d\xi = \int\limits_{\Real^n} \overline{\Fs(\varphi)(\xi)} \Fs (uw)(\xi) \bigg( \, \int\limits_{B(0,1)} ( i \xi y)^j e^{ i r y \xi } dy \bigg) d\xi. 
\]
For $r=0$ this identity reads
\[
\int\limits_{\Real^n} i^j \overline{\Fs(\varphi u)(\xi)} \Fs(w)(\xi) \bigg(\, \int\limits_{B(0,1)}  (\xi y)^j  dy \bigg) d\xi = \int\limits_{\Real^n} i^j \overline{\Fs(\varphi)(\xi)} \Fs (uw)(\xi) \bigg( \, \int\limits_{B(0,1)} (\xi y)^j  dy \bigg) d\xi. 
\]
Note that 
\[ \int\limits_{B(0,1)} (\xi y)^j dy = \int\limits_{B(0,1)} (\xi_1 y_1 + \ldots + \xi_n y_n )^j dy= \int\limits_{B(0,1)} \sum_{|\alpha|=j} \binom{\abs{\alpha}}{\alpha} \xi^\alpha y^\alpha dy = \sum_{\abs{\alpha}=j} A_\alpha \xi^\alpha . \]
Using the above observations equation \eqref{eq4} transforms to
\begin{equation}\label{eq5}
\int\limits_{\Real^n} \sum_{|\alpha|=j} A_\alpha (i\xi)^\alpha \overline{\Fs(\varphi u)(\xi)} \Fs(w)(\xi)  d\xi = \int\limits_{\Real^n}\sum_{|\alpha|=j} A_\alpha (i\xi)^\alpha \overline{\Fs(\varphi)(\xi)} \Fs (uw)(\xi)  d\xi. 
\end{equation}
Apply the Parseval identity in \eqref{eq5} and move the expression on the left-hand side to the right-hand side in order to recover the following equation
\begin{equation*}
\int\limits_{\Real^n} \sum_{|\alpha|=j} A_\alpha \varphi(x) \Big( D^\alpha u(x) w(x)  -   u(x) D^\alpha w(x) \Big)dx =0,
\end{equation*}
which is a weak formulation of the following equation
\[ \sum_{|\alpha|=j} A_\alpha \Big( D^\alpha (uw) - u D^\alpha w \Big) =0.\]
The proof of Theorem \ref{twsys} is, therefore, completed.
\end{proof}

\section{Applications of Theorem \ref{twsys} }\label{ap1}

In this section we illustrate Theorem \ref{twsys} by determining the space $\Hs(\Omega,d,dx)$ in case of the distance function~$d$ being induced by the $l^p$ norm and a constant weight $w=1$. Our goal is to show that whenever $p\ne 2$ and $n=2$, the space $\Hs(\Omega,l^p,dx)$ consists of at most 8 linearly independent harmonic polynomials. We already know that $\Hs(\Omega,l^2,dx)$ consists of all harmonic functions in $\Omega$, which differs significantly from the previous case. Moreover, we describe system \eqref{system3} for $p=2$ and smooth $w$ and compare with the equations from Theorem \ref{twbo}. Our computations are new both for $\Hs(\Omega,l^p,dx)$ with $p\ne 2$ and for $p=2$ and a smooth weight.

Let us consider the space $\Real^n$ with the distance $l^p$ for $1\leq p < \infty$ and a smooth weight $w$. First, we calculate coefficients $A_\alpha$ for $\alpha$. Due to Remark \ref{rem1} we only need to consider multi-indices $\alpha$ with even components. Using the Dirichlet Theorem (see \cite{D1}, p. 157) we obtain
\begin{equation}\label{ani}
A_\alpha =2^n \binom{\abs{\alpha}}{\alpha} \int\limits_{\substack{\sum x_i^p < 1, x_i \geq 0}} x_1^{\alpha_1} \cdot \ldots \cdot x_n^{\alpha_n} dx = \left( \frac{2}{p} \right)^n \binom{\abs{\alpha}}{\alpha} \frac{\prod\limits_{i=1}^n \Gamma\left( \frac{\alpha_i +1}{p} \right)}{\Gamma\left( \frac{|\alpha|+n+p}{p} \right) },
\end{equation}
where $\Gamma$ stands for the gamma function. Notice, that coefficients $A_\alpha$ for $j=2$ are constant by symmetry of balls in the $l^p$ norm. Therefore, the equation of system \eqref{system3} for $j=2$ translates to
\[ \sum_{i=1}^n \frac{\partial^2}{\partial x_i^2} (uw) - u \frac{\partial^2}{\partial x_i^2}(w) = 0,\]
or equivalently to 
\begin{equation}\label{eq11}
 w \Delta u +2 \nabla u \nabla w =0.
\end{equation}
Let us recall that since \eqref{eq11} is an elliptic equation with bounded coefficients, then every weak solution is smooth and solves \eqref{eq11} in a classical way. 
 Therefore $\Hs(\Omega,l^p,w) \subset C^\infty(\Omega)$ and the system \eqref{system3} can be understood in the classical sense. In order to describe further equations we need to divide our calculations into more specific instances: $p=2$, $p = \infty$ and remaining values of $1 \leq p < \infty$.

\subsection{The case of weighted $l^2$ distance}\label{ldwa}

We show that for $p=2$ system \eqref{system3} is equivalent to the~one~in~\eqref{system2}, see Theorem \ref{twbo}. We begin with computing the coefficients $A_\alpha$ in \eqref{ani}, which  take the following form \mbox{(including the case $j=2$ discussed in the beginning of Section \ref{ap1})}
\begin{equation}\label{eq9}
A_\alpha =\binom{\abs{\alpha}}{\alpha} \frac{\prod_{i=1}^n \Gamma\left( \frac{\alpha_i +1}{2} \right)}{\Gamma\left( \frac{|\alpha|+n+2}{2} \right) }.
\end{equation}
Recall the two formulas concerning the Gamma function. For any $k \in \Natural$ there holds
\[ \Gamma \left( \frac{k}{2} \right) = \sqrt{\pi}\frac{(k-2)!!}{2^{\frac{k-1}{2}}} \quad \text{and} \quad \Gamma \left( k+\frac{1}{2}\right)= \sqrt{\pi} \frac{(2k)!}{4^k k!} .\]
We use the above in \eqref{eq9} to obtain that
\[ A_\alpha =\binom{\abs{\alpha}}{\alpha} \frac{2^{\frac{j+n+1}{2}}}{\sqrt{\pi} (j+n)!!} \prod\limits_{i=1}^n \left( \sqrt{\pi} \frac{\alpha_i!}{2^{\alpha_i} ( \frac{\alpha_i}{2})!} \right) = \binom{\abs{\alpha}}{\alpha} \frac{\pi^{ \frac{n-1}{2}}  2^{\frac{n+1}{2}}}{(j+n)!!} \prod\limits_{i=1}^n \frac{\alpha_i!}{(\frac{\alpha_i}{2})!}. \]
Therefore the $j$-th equation of system \eqref{system3} can be written in the following form
\begin{align}
0&= \sum_{|\alpha|=j, \alpha_i \in 2\Natural } A_\alpha \left( D^\alpha (uw) -uD^\alpha w \right) \nonumber \\
& = \sum_{|\alpha|=j, \alpha_i \in 2\Natural}  \frac{j!}{\alpha_1! \ldots \alpha_n!} \frac{\pi^{\frac{n-1}{2}}  2^{\frac{n+1}{2}}}{(j+n)!!} \prod\limits_{i=1}^n \frac{\alpha_i!}{(\frac{\alpha_i }{2})!} (D^\alpha (uw)-u D^\alpha w) \label{eq10} \\
&= \frac{j! \pi^{\frac{n-1}{2}} 2^{\frac{n+1}{2}}}{(\frac{j}{2})! (j+n)!! } \sum_{|\beta|=j/2} 
\binom{\frac{j}{2}}{\beta} (D^{2\beta} (uw) - u D^{2\beta } w). \nonumber
\end{align}

Next, observe that for any $f \in C^{2l}(\Omega)$ its $l$-th Laplace operator can be written in the following form
\begin{equation}\label{eq12}
 \Delta^l f = \left( \frac{\partial^2}{\partial x_1^2} + \ldots +  \frac{\partial^2}{\partial x_n^2} \right)^l f =\sum_{|\beta|=l} \frac{l!}{\beta_1!\ldots \beta_n!} D_{2\beta} f,
\end{equation}
where the multinomial formula has been applied. Finally by \eqref{eq10} and \eqref{eq12} we conclude that in the $l^2$-case system \eqref{system3} is equivalent to
\begin{equation}\label{system4}
\Delta^l (uw) = u\Delta^l w, \qquad \text{for }l=1,2,\ldots
\end{equation}
In fact \eqref{system4} is equivalent to \eqref{system2}. To that end observe that $\Delta(uw) = w \Delta u+2\nabla u \nabla w + u \Delta w$. Upon joining this with the equation of \eqref{system4} corresponding to $j=2$ we obtain the first equation of \eqref{system2}. Further equations of \eqref{system2} follow from \eqref{system4} and the following computation:
\[ u \Delta^{l+1}w =\Delta(\Delta^l(uw)) = \Delta(u \Delta^l w )= \Delta u \Delta^l w + 2 \nabla u \nabla (\Delta^l w )+ u \Delta^{l+1} w. \]
Therefore 
\[ \Delta u \Delta ^l w +  2 \nabla u \nabla ( \Delta^l w) =0 \qquad \text{for }l=0,1,2,\ldots\] 
and we end this part of discussion by concluding, that by above considerations our Theorem \ref{twsys} is a generalization of Bose's result, see Theorem \ref{twbo}.

\subsection{The case of $l^p$ distance for $p \not \in \{ 2,\infty \}$}

Strongly harmonic functions on $\Omega \subset \Real^n$ equipped with the $l^p$-distance and the Lebesgue measure behave quite differently in case of $p \not \in \{ 2,\infty \}$ than for $p=2$. In what follows we demonstrate that only finitely many equations of system \eqref{system3} are nontrivial, and that in fact all of functions in $\Hs(\Omega, l^p , dx)$ are harmonic polynomials. For the sake of simplicity we consider case $n=2$, and $u$~depending on two variables $x:=x_1$ and $y:=x_2$.

We now focus our attention on equations of system \eqref{system3} for $j>2$ since the equation for $j=2$ is described in \eqref{eq11}. We examine the differential operator $R_j:=\sum_{|\alpha|=4} A_\alpha D^\alpha$. We already showed that for $p=2$ operator $R_2$ is equal to $\Delta$ up to a multiplicative constant. Recall formula \eqref{ani}:
\[ A_\alpha = \binom{\abs{\alpha}}{\alpha} \left( \frac{2}{p} \right)^2 \frac{\prod\limits_{i=1}^2 \Gamma\left( \frac{\alpha_i +1}{p} \right)}{\Gamma\left( \frac{|\alpha|+2+p}{p} \right) }.\]
We restrict our attention to the part of $A_\alpha$ varying with respect to $\alpha$, i.e. $ \prod_{i=1}^2 \Gamma \left((\alpha_i +1)/p \right)$. Let us observe, that for $|\alpha|=4$ those coefficients attain only two different values:
\begin{enumerate}[{(1)}]
\item $ \Gamma\left(\frac{5}{p} \right)\Gamma \left(\frac{1}{p}\right)$, whenever $\alpha=(4,0)$ or $\alpha = (0,4)$. This coefficient stands by $\frac{\partial^4}{\partial x^4}$ and $\frac{\partial^4}{\partial y^4}$ in $R_4$,
\item  $6\Gamma\left(\frac{3}{p}\right)^2$, if $\alpha=(2,2)$. This coefficient appears by $\frac{\partial^4}{\partial x^2 \partial y^2}$ in operator $R_4$.
\end{enumerate}
Therefore $R_4$ takes a form
\[ \left( \frac{2}{p}\right)^2 \Gamma \left( \frac{|\alpha|+4}{p}\right)^{-1} \left[\Gamma\left(\frac{5}{p} \right)\Gamma \left(\frac{1}{p}\right)\left( \frac{\partial^4}{\partial x^4} + \frac{\partial^4}{\partial y^4}\right) +6\Gamma\left(\frac{3}{p}\right)^2 \left( \frac{\partial^4}{\partial x^2 \partial y^2}\right) \right],\] 
which, up to a multiplicative constant, reduces to $\Delta^2= \frac{\partial^4}{\partial x^4} + \frac{\partial^4}{\partial y^4} + 2 \frac{\partial^4}{\partial x^2 \partial y^2}$ if and only if 
\[f(p):= \frac{\Gamma\left(\frac{3}{p}\right)^2}{\Gamma\left(\frac{5}{p} \right)\Gamma \left(\frac{1}{p}\right)} = \frac{1}{3}.\]
By the previous considerations this holds true for $p=2$. Let us differentiate $f$ with respect to $p$. Recall, that the formula for derivative of the gamma function stays:
\[\Gamma'(z) = \Gamma(z) \left(-\frac{1}{z} -\gamma -\sum_{k=1}^\infty \frac{1}{k+z}-\frac{1}{k} \right) = \Gamma(z) \Psi(z), \]
where $\gamma$ is the Euler constant and $\Psi$ is the digamma function defined by the above equality. We use this identity to compute the following
\begin{align*}
f'(p) &= \frac{2\Gamma\left( \frac{3}{p} \right)^2 \Psi\left( \frac{3}{p} \right)\left(-\frac{3}{p^2}\right)\Gamma\left( \frac{5}{p} \right) \Gamma\left( \frac{1}{p} \right)}{\Gamma\left( \frac{5}{p} \right)^2 \Gamma\left( \frac{1}{p} \right)^2}\\
&-\frac{\Gamma\left( \frac{3}{p} \right)^2\left[ \Gamma\left( \frac{5}{p} \right)\Psi\left( \frac{5}{p} \right)\left(-\frac{5}{p^2}\right)\Gamma\left( \frac{1}{p} \right) + \Gamma\left( \frac{5}{p} \right)\Gamma\left( \frac{1}{p} \right)\Psi\left( \frac{1}{p} \right)\left(-\frac{1}{p^2}\right) \right]}{\Gamma\left( \frac{5}{p} \right)^2\Gamma\left( \frac{1}{p} \right)^2} \\
&= f(p) \left( -\frac{6}{p^2} \Psi\left(\frac{3}{p}\right) + \frac{5}{p^2} \Psi\left( \frac{5}{p} \right) +\frac{1}{p^2} \Psi \left(\frac{1}{p} \right)\right).
\end{align*}
Since $f$ is positive for $p \in [1,\infty)$, we only need to investigate the sign of the second factor in the above formula:
\begin{align*}
&-\frac{6}{p^2} \Psi\left(\frac{3}{p}\right) + \frac{5}{p^2} \Psi\left( \frac{5}{p} \right) +\frac{1}{p^2} \Psi\left(\frac{1}{p}\right) \\
&= - \frac{6}{p^2}\left( -\frac{p}{3}-\gamma - \sum_{k=1}^\infty \frac{1}{k+3/p} -\frac{1}{k}\right)+ \frac{5}{p^2} \left( -\frac{5}{p}-\gamma - \sum_{k=1}^\infty \frac{1}{k+5/p}-\frac{1}{k}\right)\\
&+ \frac{1}{p^2} \left( -p-\gamma -\sum_{k=1}^\infty \frac{1}{k+1/p}-\frac{1}{k}\right) = -\frac{1}{p^2} \sum_{k=1}^\infty \frac{-6}{k+3/p}+\frac{5}{k+5/p}+\frac{1}{k+1/p}\\ &= \frac{1}{p^2}\sum_{k=1}^\infty \frac{8k}{p(k+3/p)(k+5/p)(k+1/p)} >0.
\end{align*}
Therefore $f$ is monotonically increasing on $[1,\infty)$ and attains value $1/3$ exactly at $p=2$. We conclude our computations with the following:
\[ R_4 = \begin{cases} 
\Gamma\left( \frac{5}{2}\right)\Gamma\left( \frac{1}{2}\right) \Delta^2 & \text{for } p=2, \\ 
 \Gamma\left(\frac{5}{p} \right)\Gamma \left(\frac{1}{p}\right) \left( \Delta^2 +  \left( \frac{6\Gamma\left(\frac{3}{p}\right)^2}{\Gamma\left(\frac{5}{p} \right)\Gamma \left(\frac{1}{p}\right)} -2 \right) \frac{\partial^4}{\partial x^2 \partial y^2} \right) & \text{for } p\ne 2.
\end{cases}
\]
We are now in a position to apply Theorem \ref{twsys}. Let $u \in \Hs(\Omega,l^p,dx)$. Then it satisfies the system of equations \eqref{system3}. By observation \eqref{eq11} equation of system \eqref{system3} for $j=2$ reduces to 
\begin{equation}\label{lap}
\Delta u = 0,
\end{equation}  hence $u$ is harmonic, and its bilaplacian vanishes. Moreover $u$ has to satisfy equation of system \eqref{system3} for $j=4$,~i.e. $ R_4 u =0$. Since bilaplacian of $u$ vanishes, therefore $u$ is in fact solution to $u_{x x y y}=0$. Let us observe, that differentiating twice $\Delta u$ with respect to $x$ and $y$ respectively we obtain
\[ u_{x x x x} + u_{x x y y} = 0 \quad \text{and} \quad u_{x x y y} + u_{y y y y} = 0. \]
Therefore both $u_{x x x x}=0$ and $u_{y y y y}=0$, which means that for each fixed value of $y$ function $u(x,y)$ is a polynomial in $x$ of degree at most 3 and analogously for a fixed $x$ function $u(x,y)$ is a polynomial in $y$ of degree at most 3. Then there exist $a_i(y)$ and $b_i(x)$ for $i=0,1,2,3$ such that
\begin{equation}\label{e0}
u(x,y) = a_0(y) +a_1(y) x + a_2(y) x^2 +a_3(y) x^3 = b_0(x) + b_1(x) y+b_2(x) y^2 + b_3(x) y^3.
\end{equation}
In what follows we omit writing the arguments of $a_i$ and $b_i$. Simple calculations give us that
\begin{equation}\label{e1}
u_{x x x x}=b_0^{(4)} +b_1^{(4)} y + b_2^{(4)} y^2 + b_3^{(4)} y^3 =0,
\end{equation} 
and
\begin{equation}\label{e2}
u_{y y y y}=a_0^{(4)} + a_1^{(4)} x +a_2^{(4)} x^2 +a_3^{(4)} x^3 =0.
\end{equation} 
Now at each fixed $x$ in \eqref{e1} the polynomial in $y$ has to have all coefficients equal to $0$ due to the Equality of Polynomials Theorem, hence $b_i^{(4)} = 0$ for $i=1,2,3,4$. Similarly, at \eqref{e2} we set that $a_i^{(4)}=0$ for all $i=1,2,3,4$. Therefore, all of $a_i$ and $b_i$ are polynomials of degree at most $3$. Moreover, we know that $u_{x x y y}=0$. We calculate this derivative in \eqref{e0} to get
\[ 0=u_{x x y y} =2a_2'' + 6 x a_3'' = 2b_2'' +6yb_3 ''.\]
Thus, once again we obtain that $a_i''=0$ and $b_i''=0$ for $i=2,3$, so $a_2$, $a_3$, $b_2$ and $b_3$ are in fact of degree at most $1$. By the above considerations we conclude that $u$ is a linear combination of the monomials
\begin{equation}\label{e3}
 1,x,y,xy,x^2,x^3,xy^2,xy^3,x^2y,x^3y,y^2,y^3, 
\end{equation}
which solves equation \eqref{lap}. Therefore, $u$ has to be a harmonic polynomial of the form described by \eqref{e3}. The part of $u$ generated by $\{ 1,x,y,xy\}$ is already harmonic and for that reason we only need to consider $u$ being a combination of the remaining monomials in \eqref{e3}, i.e.
\[ u=c_1 x^2 + c_2 x^3 + c_3 xy^2 +c_4 xy^3 + c_5 x^2y+c_6 x^3y + c_7 y^2 + c_8 y^3.\]
Inserting $u$ to \eqref{lap} we get the following
\[ 0=2 (c_1 + c_7) + 2x(3c_2 + c_3) + 6xy(c_4+c_6) + 2y(c_5+3 c_8), \]
and once again using the Equality of Polynomials Theorem we obtain that
\begin{equation}\label{e33}
 u \in \text{span} \left\{1,x,y,xy, x^2-y^2, xy^2-\frac{x^3}{3}, xy^3-x^3y,x^2y-\frac{y^3}{3} \right\}.
\end{equation}
Finally, let us observe that in equations of system \eqref{system3} for $j=6$ there appear only the following operators
\[ \frac{\partial^6}{\partial x^6}, \frac{\partial^6}{\partial x^4 \partial y^2}, \frac{\partial^6}{\partial x^2 \partial y^4},\frac{\partial^6}{ \partial y^6}, \]
which all vanish on $u$ as in \eqref{e33}. The triviality of equations for $j>6$ follows immediately. Therefore, we summarize our discussion with the following inclusion:
\begin{equation}\label{lp}
\Hs(\Omega,l^p,dx) \subset \text{span} \left\{1,x,y,xy, x^2-y^2, xy^2-\frac{x^3}{3}, xy^3-x^3y,x^2y-\frac{y^3}{3} \right\}. 
\end{equation}
We postpone the proof of the opposite inclusion till Section \ref{51}. Similar observation can be made in higher dimensions.

\subsection{The case of $l^p$ distance for $p =\infty$}

In order to complete our illustration of Theorem \ref{twsys} we need to consider the remaining case, i.e. characterize functions $u$ in $\Hs(\Omega,l^p,dx)$ for $p=\infty$ by using Theorem \ref{twsys}. In this case $B(0,1)=[-1,1]^n$ in $l^\infty$ norm. Therefore, we obtain the following formula for the coefficients $A_\alpha$ in \eqref{system3}:
\[ A_\alpha =\binom{\abs{\alpha}}{\alpha} \int_{-1}^1 x_1^{\alpha_1} \cdot \ldots \int_{-1}^1 x_n^{\alpha_n} = \binom{\abs{\alpha}}{\alpha}\frac{2^n}{\prod_{i=1}^n (\alpha_i +1)}.\]
Then, after inserting $A_\alpha$ and dividing by the $2^n $ factor, system \eqref{system3} converts to the following
\begin{equation}
\sum_{|\alpha|=j,2|\alpha_i} \binom{\abs{\alpha}}{\alpha} \frac{1}{(\alpha_1+1)! \cdot \ldots \cdot (\alpha_n+1) !}  D^\alpha u =0.
\end{equation}
As in the previous subsection we restrict our attention to case $n=2$ and write out the equation for $j=2$: $\frac{1}{6} (u_{x x} + u_{y y})=0$. Hence $u$ is a harmonic function. Equation for $j=4$ is the following
\[\frac{1}{120} (u_{xxxx}+u_{yyyy}) + \frac{1}{6} u_{xxyy} =0, \]
and can be reduced to $\Delta^2 u +  20 u_{xxyy}=0$. This, combined with an analogous discussion to the one ending the previous subsection leads us to the conclusion that inclusion \eqref{lp} holds true also for $p=\infty$.

\section{The converse of Theorem \ref{twsys}}

Since both Theorems \ref{F-L}, \ref{twbo} and Proposition \ref{bosod} give not only the necessary, but also the sufficient condition for the mean value properties in the sense of Flatto and Bose, respectively, our next goal is to find an appropriate counterpart of these results. In case of nonconstant weights Proposition \ref{bosod} imposes an additional PDE condition on $w$, hence we expect an analogous condition. From the point of view of our further considerations, the following \emph{generalized Pizzetti formula} introduced by Zalcman in \cite{F}, will be vital.

\begin{theorem}[Theorem 1, \cite{F}]\label{twza}
Let $\mu$ be a finite Borel measure on $\Real^n$ with compact support and \sloppy ${\Fs(\xi)=\int_{\Real^n}e^{-i(\xi y)}d\mu(y)}$ be the Fourier transform of the measure $\mu$. Suppose that $f$ is an analytic function on a domain $\Omega \subset \Real^n$. Then the following equality holds
\begin{equation}\label{gpiz}
 \int_{\Real^n} u(x+ry) d\mu(y) = [\Fs(-rD)f](x),
\end{equation}
for all $x \in \Omega$ and $r>0$ such that the left-hand side exists and the right-hand side converges. The symbol~$D$ is given by $D:=-i \left( \frac{\partial}{\partial x_1},\ldots,\frac{\partial}{\partial x_n} \right)$.
\end{theorem}

%

\begin{remark}
Formula \eqref{gpiz} is one of the most important tools used in the proof of our Theorem \ref{twod} and requires the analyticity of functions. Thus we need to assume analyticity of weight $w$. Weakening this assumption would require developing a completely new approach.
\end{remark}

\begin{remark}
Theorem \ref{twod} generalizes the converse part of Theorem \ref{F-L} and Proposition \ref{bosod}. Theorem \ref{F-L} follows from our Theorem \ref{twod} with $w=1$ and the regularity of solutions to the Laplace equation \eqref{lap}. Proposition~\ref{bosod} is generalized by Theorem \ref{twod} due to the following lemma.
\begin{lemma}
Suppose that $w \in C^{2l}(\Omega)$ solves the following equation
\begin{equation}\label{gef} 
\Delta^l w + a_{l-1} \Delta^{l-1} w + \ldots + a_1 \Delta w + a_0 w = 0,
\end{equation}
where all $a_i \in \mathbb{C}$. Then $w$ is analytic in $\Omega$.
\end{lemma}
\begin{proof}
We prove the lemma by the mathematical induction with respect to $l$. Recall the following fact (see Theorem 2 on p. 229 in \cite{E2}): Suppose that $w \in C^2(\Omega)$ solves the following equation
\begin{equation}\label{daur}
L w + \lambda w = \varphi, 
\end{equation}
where $L$ is elliptic and $\varphi$ is analytic in $\Omega$, $\lambda \in \mathbb{C}$. Then $w$ is analytic in $\Omega$.

If $l=1$, then we use the above regularity fact with $a_0=\lambda$ and $\varphi \equiv 0$. Now let us assume that the assertion holds for $l-1$ and consider $w$ as in \eqref{gef}. We may rewrite this equation as follows with $\lambda \in \mathbb{C}$:
\begin{align*}
0&=\Delta^{l-1} (\Delta w + \lambda w) +(a_{l-1}-\lambda) \Delta^{l-2} (\Delta w + \lambda w ) + (a_{l-2}-\lambda(a_{l-1}-\lambda)) \Delta^{l-3}(\Delta w + \lambda w) + \ldots \\
&+ (a_1 - \lambda(a_2-\lambda(\ldots) ) ) (\Delta w + \lambda w) + \Big(a_0 -\lambda(a_1 - \lambda(a_2-\lambda(\ldots) ) )\Big) w.
\end{align*}
Choose such $\lambda $, so that the last factor, standing by $w$ in the equation, vanishes. We use the assumption for $l-1$ to obtain that $\Delta w + \lambda w$ is an analytic function, denoted by $\varphi$, i.e. $ \Delta w + \lambda w = \varphi.$ 
This observation together with the regularity observation allow us to conclude the proof.
\end{proof}
\end{remark}

Before proving Theorem \ref{twod} we need to show the following regularity result for strongly harmonic functions.
\begin{lemma}\label{lem4}
Suppose that  $\Omega \subset \Real^n$ is open, $w$ is a positive analytic function and $d$ is induced by norm. Then any $u \in \Hs(\Omega,d,wdx)$ is analytic as well.
\end{lemma}

\begin{proof}
By the assumptions function $u$ is a weak solution to equation for $j=2$ of system \eqref{system3}. Let us show that this equation is strongly elliptic. Let $\xi \in \Real^n$, then
\begin{align*}
\sum_{|\alpha|=2} A_\alpha w \xi^\alpha =w \int_{B(0,1)}  (x \cdot \xi)^2 \geq  w \int_{\|x\|_2 \leq \varepsilon } (x \cdot \xi)^2, 
\end{align*} 
where the last estimate holds with some $\varepsilon>0$ since $d$ is equivalent to the Euclidean distance. Next, observe that
\[\int_{\|x\|_2 \leq \varepsilon } (x \cdot \xi)^2 = \int_{\|x\|_2 \leq \varepsilon} \cos^2 \angle( x,\xi) \|x\|_2^2 \|\xi\|_2^2 = \|\xi\|_2^2 \int_{\|x\|_2 \leq \varepsilon} \cos^2 \angle( x,\xi) \|x\|_2^2 = \theta \|\xi\|_2^2, \]
where $\theta>0$ is defined by the above equality. It does not depend on $\xi$ due to the symmetry of the Euclidean ball. Hence
\[ \sum_{|\alpha|=2} A_\alpha w \xi^\alpha  \geq \theta w \|\xi\|_2^2.\]
Therefore the operator $\sum_{|\alpha|=2} A_\alpha w D^\alpha$ is strongly elliptic. We apply the regularity result \eqref{daur} to obtain that $u$ is analytic.
\end{proof} 


Now we are in a position to prove Theorem \ref{twod}.

\begin{proof}[Proof of Theorem \ref{twod}]
We need to show the following equality
\begin{equation}\label{equ8}
 u(x) \int_{B(0,1)} w(x+ry) dy = \int_{B(0,1)} u(x+ry)w(x+ry) dy.
\end{equation}
In order to prove \eqref{equ8} we use the generalized Pizzetti formula for a measure $\mu$ being the normalized Lebesgue measure on the unit ball. Then
\[ \Fs(\xi) = \int_{B(0,1)} e^{-i \xi y} dy = \sum_{j=0}^\infty \frac{(-i)^j}{j!} \int_{B(0,1)} (\xi y )^j dy = \sum_{j=0}^\infty \frac{(-i)^j}{j!} \sum_{|\alpha|= j }A_\alpha \xi^\alpha = \sum_{\alpha \in \Natural^n} \frac{(-i)^{|\alpha|}}{|\alpha|!} A_\alpha \xi^\alpha,\]
where $A_\alpha =\binom{|\alpha|}{\alpha} \int_{B(0,1)} y^\alpha dy$. We apply Theorem \ref{twza} twice: to $w$ and $uw$ to obtain
\begin{align}
\label{e43} \int_{B(0,1)} w(x+ry) dy &= \sum_{\alpha \in \Natural^n} \frac{r^{|\alpha|}}{|\alpha|!} A_\alpha D^\alpha w(x),\\
 \label{e44} \int_{B(0,1)} u(x+ry) w(x+ry) dy& = \sum_{\alpha \in \Natural^n} \frac{r^{|\alpha|}}{|\alpha|!} A_\alpha D^\alpha (u(x)w(x)),
\end{align}
Subtract the above equations from each other to obtain the following:
\begin{align*}
& u(x)  \int_{B(0,1)} w(x+ry) dy-\int_{B(0,1)} u(x+ry) w(x+ry) dy \\
&= \sum_{\alpha \in \Natural^n} \frac{r^{|\alpha|}}{|\alpha|!} A_\alpha \left( u(x) D^\alpha(w(x)) - D^\alpha (u(x)w(x))  \right) \\
&= \sum_{j=0}^\infty \frac{r^j}{j!} \sum_{|\alpha|=j}  A_\alpha \left( u(x) D^\alpha(w(x)) - D^\alpha (u(x)w(x))  \right) =0,
\end{align*}
where in the last step we appeal to \eqref{system3}. Thus, the proof is completed.
\end{proof}

\subsection{Applications of Theorem \ref{twod}}\label{51}

In the last part of our work we present some of the consequences of Theorem \ref{twod}. First of all, we are now in a position, to augment observation \eqref{lp} with the converse inclusion:
\[ \Hs(\Omega, l^p,dx) =\text{span} \left\{1,x,y,xy, x^2-y^2, xy^2-\frac{x^3}{3}, xy^3-x^3y,x^2y-\frac{y^3}{3} \right\}, \]
and to add, that now we see that the dimension of $\Hs(\Omega, l^p,dx)$ is equal to 8. The result in dimension 3 is due to Grzegorz Łysik, who computed it to be equal to 48. Those numbers coincide with $2^n n!$ - the number of linear isometries of $(\Real^n,l^p)$, which is discovered in \cite{LI}. We believe that there is a link between the dimension of the space $\Hs(\Omega, d,\mu)$, still to be examined.

Moreover, let us consider the metric measure space $(\Omega,l^2,wdx)$ for the analytic weight function $w$. Then, by Theorems \ref{twsys} and \ref{twod} and Section \ref{ldwa} we know that $\Hs(\Omega,l^2,wdx )$ consists exactly of solutions to the following system of equations
\begin{equation}\label{e5}
\Delta u \Delta^j w +2 \nabla u \nabla (\Delta^j w)=0, \qquad \text{for } j=0,1,\ldots.
\end{equation}
Let us observe, that $u$ solves also infinitely many other systems of equations, obtained from \eqref{e5} by excluding $l \in \Natural$ initial equations 
\[ \Delta u \Delta^{j+l} w +2 \nabla u \nabla (\Delta^{j+l} w)= \Delta u \Delta^j(\Delta^l w) +2 \nabla u \nabla (\Delta^j(\Delta^l w))=0, \qquad \text{for } j=0,1,\ldots. \]
Therefore, $u$ is strongly harmonic in countably many metric measure spaces $(\Omega,l^2,\Delta^l w dx)$ for all $l \in \Natural$. In other words, function $u$ has infinitely many mean value properties, with respect to different weighted Lebesgue measures $d\mu = \Delta^l w dx$ for all $l \in \Natural$, whenever $\Delta^l w$ are positive.
\vspace{10pt}

\bibliographystyle{amsplain}

\end{document}
